\documentclass[12pt,oneside]{amsart}
\usepackage{xspace}
\usepackage{graphicx}
\usepackage{latexsym}
\usepackage{mathtools}
\usepackage{amsfonts,amsmath,amscd,amssymb}

\usepackage{hhline}
\usepackage{multirow}
\usepackage{framed}
\usepackage{mathrsfs}
\usepackage{tikz-cd}
\usepackage{enumitem}
\usepackage{hyperref}
\usepackage[normalem]{ulem}

\calclayout

\DeclareMathOperator{\eva}{ev}
\DeclareMathOperator{\id}{id}

\DeclareMathOperator{\Hom}{Hom}
\DeclareMathOperator{\Iso}{Iso}
\DeclareMathOperator{\Mon}{Mono}
\DeclareMathOperator{\HO}{Ho}

\DeclareMathOperator{\GL}{GL}
\DeclareMathOperator{\SL}{SL}
\DeclareMathOperator{\SO}{SO}
\DeclareMathOperator{\Sp}{Sp}

\DeclareMathOperator{\U}{U}
\DeclareMathOperator{\mO}{O}
\DeclareMathOperator{\Aut}{Aut}

\DeclareMathOperator{\ResOp}{\downarrow}
\DeclareMathOperator{\IndOp}{\uparrow}

\newcommand{\catname}[1]{{\normalfont\textbf{#1}}}
\newcommand{\res}[2]{\ResOp_{#1}^{#2}}
\newcommand{\ind}[2]{\IndOp_{#1}^{#2}}

\newcommand{\Mod}[1]{\catname{Mod$_{#1}$}}

\newcommand{\G}{\widehat{G}}

\newcommand{\R}{\mathbb{R}}

\newcommand{\C}{\mathbb{C}}

\newcommand{\fn}[3]{#1 : #2 \rightarrow #3}
\newcommand{\twopartdef}[4]
{
	\left\{
		\begin{array}{ll}
			#1 & \mbox{if } #2 \\
			#3 & \mbox{if } #4
		\end{array}
	\right.
}

\newcommand{\bC}{{\mathbb C}}

\newcommand{\bF}{{\mathbb F}}

\newcommand{\bH}{{\mathbb H}}

\newcommand{\bK}{{\mathbb K}}

\newcommand{\bR}{{\mathbb R}}

\newcommand{\sA}{{\mathcal A}}

\newcommand{\sC}{{\mathcal C}}

\newcommand{\sH}{{\mathcal H}}

\newcommand{\sL}{{\mathcal L}}

\newcommand{\sR}{{\mathcal R}}
\newcommand{\sT}{{\mathcal T}}

\newcommand{\Id}{{\ensuremath{\mathrm{Id}}}}

\newcommand{\ve}{\ensuremath{\mathbf{e}}\xspace}

\newcommand{\vg}{\ensuremath{\mathbf{g}}\xspace}

\newcommand{\vk}{\ensuremath{\mathbf{k}}\xspace}

\newcommand{\vv}{\ensuremath{\mathbf{v}}\xspace}
\newcommand{\vw}{\ensuremath{\mathbf{w}}\xspace}

\newcommand{\vz}{\ensuremath{\mathbf{z}}\xspace}

\newtheorem{thm}{Theorem}
\numberwithin{thm}{section}
\newtheorem{lem}[thm]{Lemma} 
\newtheorem{prop}[thm]{Proposition} 
\newtheorem{cor}[thm]{Corollary}

\theoremstyle{definition}
\newtheorem{defn}[thm]{Definition}

\theoremstyle{remark}

\begin{document}
\title[Real Representations]{Real Representations of $C_2$-Graded Groups: The Linear and Hermitian Theories}
\author{Dmitriy Rumynin}
\email{D.Rumynin@warwick.ac.uk}
\address{Mathematics Institute, University of Warwick, Coventry, CV4 7AL, UK
  \newline
\hspace*{0.31cm}  Associated member of Laboratory of Algebraic Geometry, National
Research University Higher School of Economics, Russia}
\author{James Taylor}
\email{james.h.a.taylor@outlook.com}
\address{Mathematical Institute, University of Oxford, Oxford, OX2 6GG, UK}
\date{September 29, 2020}
\thanks{The first author was partially supported within the framework of the HSE University Basic Research Program and the Russian Academic Excellence Project `5--100'.}
\thanks{We are indebted to Matthew B. Young for useful conversations and interest in our work.  We would like to thank Karl-Hermann Neeb for bringing the work of Wigner to our attention. We are grateful to John Jones for tutoring us in Homotopic Topology.}

\begin{abstract}
  We study linear and hermitian representations of finite $C_2$-graded groups.
  We prove that the category of linear representations is equivalent to a category of antilinear representations as an $\infty$-category. 
  We also prove that the category of hermitian representations, as an $\infty$-category, is equivalent to a category of usual representations.
\end{abstract}

\maketitle


Real representations first appeared in Quantum Mechanics in the works of Wigner \cite{WIG}.
Independently, they were introduced by Atiyah and Segal \cite{EKC} and Karoubi \cite{KAR} in the context of equivariant KR-theory.
Over time they have been actively studied by many scientists with a range of backgrounds (cf. \cite{Bor,Dimmock,Dyson,Fok,NEEB,LTJ,ROZ,BR2,R2R}).
The present paper is a sequel to our study of antilinear representations \cite{AREP}.
Here we investigate {\em linear} and {\em hermitian} representations, introduced by Young \cite{R2R}.

A $C_2$-graded group is a pair of finite groups, $G \leq \G$, where $G$ is an index 2 subgroup of $\G$.
A {\em Real} representation of $G$ is a complex representation $(V, \rho)$ of $G$ together with ``an action'' of the other coset $\G\setminus G$ satisfying appropriate algebraic coherence conditions.
In the antilinear theory,  each element $\vw \in \G \setminus G$ acts by an antilinear operator, or simply a linear map $\fn{\rho(\vw)}{\overline{V}}{V}$.
In the linear theory,  $\vw \in \G \setminus G$ acts by a bilinear form, regarded as a linear map $\fn{\rho(\vw)}{V^*}{V}$.
Finally, in the hermitian theory,  an element $\vw \in \G \setminus G$ acts by a sesquilinear form, regarded as a linear map $\fn{\rho(\vw)}{\overline{V}^*}{V}$.

The goal of the present paper is to describe the linear and hermitian categories fully.
These categories are not $\bR$-linear. Instead they are topological.
Moreover, they are ``homotopically equivalent'' to some non-full subcategories of the categories of antilinear or usual representations.
Let us now provide a detailed description of the content of the present paper. 

We start by reminding the reader the basic notions of $\infty$-categories in Section~\ref{s1}.
Then we introduce all the categories that we study, and finish the section with the statement of the main result of this paper, Theorem~\ref{main_theorem}.

In section~\ref{s2} we prove the first two parts of the main theorem, describing the two versions of the linear theory.
The first statement reduces to homotopy equivalences between certain Lie groups and their maximal compact subgroups (see Table~\ref{table_i}).
Similarly, the second statement to homotopy equivalences between homogeneous spaces.

Section~\ref{s3} is devoted to the proof the last two parts of the main theorem, describing the hermitian theory.
The proofs are parallel to the first two parts with different Lie groups and homogeneous space appearing (see Table~\ref{table_iii}). 

Finally, in Section~\ref{s4} we discuss some generalisations, outlining directions for future research.

\section{Categories} \label{s1}
\subsection{$\infty$-Categories} \label{s1.1}
Let $\sT$ be the closed monoidal category of compactly generated weakly Hausdorff topological spaces together with its Quillen model structure \cite[App. A]{Sch} (cf. \cite{HJR}).
{\em A topological category} is a category $\sC$, enriched in $\sT$.

Given a topological space $Z\in\sT$, by $[\![ Z ]\!]$ we denote the corresponding object in the homotopy category $\HO(\sT)$.
The homotopy category $\HO(\sT)$ is closed monoidal.
By $[\![ \sC ]\!]$ we denote the category with the same objects as $\sC$ and new morphisms
$$
[\![ \sC ]\!] (X,Y) \coloneqq
[\![ \sC (X,Y) ]\!]. 
$$
Since the morphisms are not sets, it is not a category in the usual sense. Instead it is a category enriched in $\HO(\sT)$.

In this paper by {\em $\infty$-categories} we understand categories of the form $[\![ \sC ]\!]$, enriched in $\HO(\sT)$, coming from topological categories. A functor (or an equivalence) of $\infty$-categories is just a functor (an equivalence) of categories enriched in $\HO(\sT)$.
This rather restrictive view of {\em $\infty$-categories}, outlined by Lurie \cite[Def. 1.1.1.6]{Lur}, is sufficient for our ends.

Let $\sC$ be a topological category. By $\Mon(\sC)$ and $\Iso(\sC)$ we denote the monomorphism and isomorphism categories of $\sC$. They have the same objects as $\sC$ but fewer morphisms
\begin{equation} \label{hom_sets}
\begin{aligned}
\Mon (\sC ) (X,Y) &\coloneqq \vk (\{ f\in \sC (X,Y) \,\mid\, f \mbox{ is a monomorphism} \}), \\
\Iso (\sC ) (X,Y) &\coloneqq \vk (\{ f\in \sC (X,Y) \,\mid\, f \mbox{ is an isomorphism} \}).
\end{aligned}
\end{equation}
A subset of $\sC (X,Y)$, equipped with the subspace topology, is weakly Hausdorff \cite[Prop A.4]{Sch} but not necessarily compactly generated. Hence, we apply the kellification functor $\vk$ to the subspace: the closed subsets of $\vk (Z)$ are compactly closed subsets of $Z$, i.e., those subsets $A\subseteq Z$ that $f^{-1} (A)$ is closed in $K$ for any compact $K$ and any continuous map $f:K\rightarrow Z$ \cite[A.1]{Sch}.

Thus, both $\Mon (\sC )$ and $\Iso (\sC )$ are topological categories.
The category $\Iso (\sC )$ is often called {\em the core} of the category $\sC$.

\subsection{Modules} \label{s1.2}
For $A$ an associative algebra over $\bR$ or $\bC$, write
$\Mod{fd}(A)$ for its category of finite-dimensional modules.
Each hom-set in $\Mod{fd}(A)$ is a finite-dimensional vector space over $\bR$.
Considering it in its Euclidean topology yields a topological category structure on $\Mod{fd}(A)$.

By a $C_2$-graded group we understand an exact sequence of finite groups
$$1\rightarrow G \rightarrow \G \xrightarrow{\pi} C_2 =\{\pm 1\} \rightarrow 1\, . $$
An element $\vg\in\G$ is called even (odd) if $\pi (\vg) =1$ ($\pi (\vg) =-1$). Two associative algebras are related to it: the complex group algebra $\bC \G$ and the skew group algebra  $\bC\ast \G$, where the coset $\G\setminus G$ acts on $\C$ by complex conjugation. We study the corresponding topological categories:
$$
\sR (G)\coloneqq \Mod{fd}(\bC \G) \ \mbox{ and } \  \sA (G)\coloneqq \Mod{fd}(\bC\ast \G) \, . 
$$ 
The category of $\sR (G)$ is the category of representations of $\G$.
Following on from our previous work \cite{AREP}, we think of the $C_2$-graded group $\G$ as a {\em Real} structure on $G$ and then of $\sA (G)$ as the category of antilinear Real representations of $G$, which is the reason to keep $G$ but not $\G$ in the notation.  
Similarly to Section~\ref{s1.1}, we are interested in the categories of isomorphisms and monomorphisms of $\sA (G)$ and $\sR (G)$.
Since the original categories are abelian, the monomorphisms are precisely the injective maps and the isomorphisms are precisely the bijective maps.
Note that the kellification functor in~\eqref{hom_sets} does not change the topology for the four new categories.
Indeed, 
the hom-set $\Mon (\sR (G)) (X,Y)$ is open in $\sR (G) (X,Y)$, hence, first countable, while every first countable topological space is compactly generated \cite[A.2]{Sch}. Similarly for the other three topological categories.

\subsection{Hermitian Representations Over a Ring} \label{s1.3}
Let $\bK=(\bK, \iota)$ be a commutative ring with involution $\iota (a)=\overline{a}$, which is allowed to be trivial.
Let $\Mod{fgp}(\bK G)$ be the category of representations of $G$ over $\bK$: we define it as the full subcategory of $\Mod{}(\bK G)$ consisting of objects that are finitely generated projective $\bK$-modules.

We need adjectives describing sesquilinear forms (cf. \cite[2.7]{AREP}). Let $\vw\in\G\setminus G$. A sesquilinear form $B:V\times V \rightarrow \bK$ on a $V\in\Mod{fgp}(\bK G)$ is called 
\begin{equation} \label{invariant_form}
\begin{aligned}
\vw\mbox{-invariant if } B(\vg u, \vw \vg \vw^{-1} v) &= B(u,v) \mbox{ for all } \vg \in G, \; u,v\in V\, , \\
\vw\mbox{-skew-hermitian if } B(u, \vw^{2} v) &= -\overline{B(v,u)} \mbox{ for all }  u,v\in V\, , \\
\mbox{and } \vw\mbox{-hermitian if } B(u, \vw^{2} v) &= \overline{B(v,u)} \mbox{ for all }  u,v\in V\, .
\end{aligned}
\end{equation}
These properties do not depend on a particular choice of $\vw$: a $\vw$-invariant form is $\vv$-invariant for any $\vv\in\G\setminus G$, etc.
If the involution is known to be trivial, we routinely use the words $\vw$-symmetric and $\vw$-alternating instead of $\vw$-hermitian and $\vw$-skew-hermitian.

By $V^*$ we denote the dual module, by $\overline{V}$ -- the conjugate module.
Let
\[
\fn{\eva}{V}{V^{**}= \overline{\overline{V}^{*}}^{*}}, \; \eva_v(f) = f(v),
\]
be the canonical isomorphism of $V$ with its double dual $\bK$-module. For a $\bK$-linear map $\fn{f}{V}{W}$, we denote  the conjugate-transpose map $\fn{f^{\ast}=\overline{f}^*}{\overline{W}^*}{\overline{V}^*}$.
A convenient notation is 
$$
(\prescript{\epsilon}{}{V}, \prescript{\epsilon}{}{f})\coloneqq \twopartdef{(V,f)}{\epsilon = 1,}{(\overline{V}^*,\overline{f}^*)}{\epsilon = -1,} \text{  and	 } \delta_{x,y,-1} \coloneqq \twopartdef{1}{x=y=-1,}{0}{$otherwise$.}
$$
\begin{defn}
A \emph{hermitian $\bK$-representation} of a $C_2$-graded group $\G$ (or a Real group $G$) is a finitely generated projective $\bK$-module $V$ with invertible linear maps $\fn{\rho(\vz)}{\prescript{\pi(\vz)}{}{V}}{V}$ for all $\vz \in \hat{G}$, such that $\rho(\ve ) = \id_V$, and
\begin{equation}\label{coher1}
\rho(\vz_2\vz_1) = \rho(\vz_2) \circ \prescript{\pi(\vz_2)}{}{\rho(\vz_1)}^{\pi(\vz_2)} \circ \eva^{\delta_{\pi(\vz_1),\pi(\vz_2),-1}}.
\end{equation}
\end{defn}

In other words, each odd element $\vw$ defines a non-degenerate sesquilinear form
$$
\fn{B_\vw}{V \times V}{\bK}, \  B_\vw(u,v) := \rho(\vw)^{-1}(v)(u)\, .
$$
If $V$ is a free $\bK$-module, we can choose a basis of $V$ and write the bilinear forms as matrices:
$$ B_{\vw}(u,v) = \underline{v}^T B(\vw) \underline{u} \ \mbox{ for all } \ u,v\in V\, ,$$
where $\underline{u}$ is the coordinate column of $u$.
We can also write the linear maps $\rho (\vz)$ as matrices $M(\vz)$, using the dual basis of $V^*$ for odd $\vw$.
Note that $M(\vw) = B(\vw)^{\flat}$ where $A^{\flat} \coloneqq (A^{-1})^*$ for an invertible matrix $A$ . 
It is instructive to write \eqref{coher1} as four different conditions depending on the parity of elements.
The condition~\eqref{coher1} for two even elements means that $(V,\rho)$ is a representation of $G$.
The other three corners tell us that
\begin{center}
\begin{tabular}{ c|c|c } 
  even-odd & odd-even & odd-odd \\
  $B_{\vg \vw}(u,v) = B_\vw(u,\vg^{-1}v)$ & $B_{\vw \vg}(u,v) = B_\vw(\vg u,v)$ & $\overline{B_{\vw_1}(u,v)} = B_{\vw_2}((\vw_1\vw_2)^{-1}v,u)$ \\
  $B(\vg\vw)=M(\vg)^{\flat} B(\vw)$ & $B(\vw\vg)=B(\vw) M(\vg)$ & $M(\vw_1\vw_2)^{\flat}=B(\vw_1) B(\vw_2)^{\flat}$ \\
  $M(\vg\vw)=M(\vg) M(\vw)$ & $M(\vw\vg)=M(\vw) M(\vg)^{\flat}$ & $M(\vw_1\vw_2)=M(\vw_1) M(\vw_2)^{\flat}$ \\ 
\end{tabular}
\end{center}
for all $\vg \in G, \vw_1,\vw_2, \vw \in \G \setminus G$, $u,v \in V$. The following useful reformulation of these conditions is easy to prove:
\begin{lem} \label{lemma_3corners}
  Suppose that $(V, \rho)\in\Mod{fgp}(\bK G)$, and we are given a non-degenerate sesquilinear form $B_\vw$ for each $\vw\in\G\setminus G$. Define $\rho(\vw) : \overline{V}^* \rightarrow V$ for $\vw\in\G\setminus G$ by,
  \[
  \rho(\vw)^{-1}(v)(u) := B_\vw(u,v).
  \]
Then these data extend $V$ to a hermitian $\bK$-representation of $\G$ if and only if
  each $B_\vw$ is $\vw$-invariant and $\vw$-hermitian, and $B_{\vw_1}(u,v) = B_{\vw_2}(u, \vw_2 \vw_1^{-1} v)$ for all $\vw_1,\vw_2\in\G\setminus G$.
\end{lem}
  
There are two competing notions of a homomorphism.
Let $(V,\rho)$, $(W, \mu)$ be two hermitian $\bK$-representations of $G$.
By  \emph{a homomorphism} of hermitian representations we understand a homomorphism of $\bK G$-modules $\fn{f}{V}{W}$ that preserves all the forms $B_\vw$:
$$
B_{\vw}^V(u,v) = B_{\vw}^W (f(u),f(v)) \ \mbox{ for all } \ u,v\in V, \ \vw \in \G\setminus G \, .
$$
Thanks to the odd-odd corner condition, this is equivalent to preserving one of the forms $B_\vw$.
Since the forms $B_\vw$ are non-degenerate, all homomorphisms are injective. We denote the resulting category $\sH_{(\bK,\iota)} (G)$. 

Now by \emph{a strong homomorphism} we understand  a $\bK$-linear map $\fn{f}{V}{W}$ such that the squares
\[
\begin{tikzcd}
    V \arrow{r}{f} \arrow[swap]{d}{\rho(\vg)} & W \arrow{d}{\mu(\vg)} \\
    V \arrow[swap]{r}{f} & W
  \end{tikzcd} \qquad \qquad
  \begin{tikzcd}
    \overline{W}^* \arrow{r}{f^*} \arrow[swap]{d}{\mu(\vz)} & \overline{V}^* \arrow{d}{\rho(\vz)} \\
    W & V \arrow{l}{f}
  \end{tikzcd}
\]
commute for all $\vg \in G$ and $\vz \in \G \setminus G$.
Clearly, a homomorphism $f$ is a strong homomorphism if and only if $f$ preserves the dual forms $B_\vw^*$. 
This is also equivalent to $f$ being bijective.
This is further equivalent to $f$ being an isomorphism. Thus, the category of hermitian $\bK$-representations with strong isomorphisms is just $\Iso (\sH_{(\bK,\iota)} (G))$. 

Observe that there is a natural notion of ``direct sum'' $(V,\rho)\oplus (W,\mu)$ of hermitian $\bK$-representations: it is a direct sum of $\bK G$-modules with obvious extension of the forms $B_\vw (v_1+w_1,v_2+w_2) \coloneqq B_\vw (v_1, v_2) + B_\vw (w_1,w_2)$. The quotation marks are justified by the limited categorical properties of this construction: ``the direct sum'' is a coproduct (but not a product) in $\sH_{(\bK,\iota)} (G)$ and has no categorical properties in $\Iso (\sH_{(\bK,\iota)} (G))$. 

\subsection{Maschke's Theorem}
Suppose that $\bK$ is a field. The involution $\iota$ could still be trivial or non-trivial.
This allows to take orthogonal complements $V = W \oplus W^{\perp}$ on finite-dimensional modules with hermitian forms as soon as the restriction of the form to $W$ is non-degenerate.

Consider $(V,\rho) \in \sH_{(\bK,\iota)} (G)$. We call a vector subspace $W\subseteq V$ a \emph{subrepresentation}, if it is a $\bK G$-submodule and all restrictions $B^V_\vw|_W$ are non-degenerate. These restrictions define a structure of hermitian representation on $W$ such that the embedding $W\hookrightarrow V$ is a morphism in $\sH_{(\bK,\iota)} (G)$.

\begin{prop} \label{Maschke_th} (Maschke's Theorem) 
  Suppose that $\bK$ is a field and $(V,\rho) \in \sH_{(\bK,\iota)} (G)$.
  If $(W,\mu)$ is a subrepresentation of $V$, then the right orthogonal complement $W^{\perp}$ under $B^V_\vw$ is a Hermitian subrepresentation with $W^{\perp} \oplus W = V$. Moreover, $W^{\perp}$ does not depend on the choice of odd $\vw$. 
\end{prop}
\begin{proof}
  Since $B^V_{\vw}|_{W} = B^W_{\vw}$, $W \cap W^{\perp} = 0$. By non-degeneracy of $B^V_\vw$, $W \oplus W^{\perp} = V$.

  The odd-odd version of \eqref{coher1} implies that $W^{\perp}$ is independent of $\vw$. 
  The even-odd version of \eqref{coher1} implies that $W^{\perp}$ is a $\bK G$-submodule.
  Finally, we can equip $W^{\perp}$ with the action of odd elements by $B^{W^{\perp}}_{\vw}\coloneqq B^V_{\vw}|_{W^\perp}$.
\end{proof}

We say that $V \in \sH_{(\bK,\iota)} (G)$ is irreducible if $V\neq 0$ and $0$, $V$ are the only subrepresentations of $V$. 
\begin{cor} (Krull-Remak-Schmidt Theorem) \label{Maschke_cor}
Every  $V \in \sH_{(\bK,\iota)} (G)$ decomposes as a finite direct sum of irreducible hermitian representations in a unique way up to permutation and isomorphism.
\end{cor}
\begin{proof}
  The decomposition easily follows from Proposition~\ref{Maschke_th}.

  Suppose that $V=V_1 \oplus \ldots \oplus V_m=W_1 \oplus \ldots \oplus W_n$ are two decompositions.
Uniqueness is proved by induction on $m$. 
One of the maps $V_1 \hookrightarrow V \twoheadrightarrow W_j$ must be an isomorphism in $\sH_{(\bK,\iota)} (G)$.
This is the induction base.
Furthermore, this gives an isomorphism between $V_1^{\perp}$ and $W_j^{\perp}$, which is the induction step.
\end{proof}

\subsection{Statement of the Main Theorem} \label{s1.4}
The most interesting field for us are the complex numbers $\C$. It has two natural involutions.
The first involution is complex conjugation. The corresponding category is denoted $\sH (G) \coloneqq \sH_{(\C,\iota)} (G)$.
We call these representations hermitian or simply H-representations.
We denote the hom-sets in this category by $\Hom_H(V,W)\coloneqq \sH (G) (V,W)$ and $\Aut_H(V)\coloneqq \sH (G) (V,V)$.

The second involution is trivial. The corresponding category is denoted $\sL (G) \coloneqq \sH_{(\C,\Id)} (G)$.
We call these representations linear or simply L-representations.
We denote the hom-sets in this category by $\Hom_L(V,W)\coloneqq \sL (G) (V,W)$ and $\Aut_L(V)\coloneqq \sL (G) (V,V)$.

\begin{thm}\label{main_theorem} 
   Let $\G$ be a finite $C_2$-graded group. The following pairs of $\infty$-categories are equivalent:
  \begin{itemize}
  \item[(i)] $[\![\Iso (\sA (G))]\!]$ and $[\![\Iso (\sL (G))]\!]$, 
  \item[(ii)] $[\![\Mon (\sA (G))]\!]$ and $[\![\sL (G)]\!]$, 
  \item[(iii)] $[\![\Iso (\sR (G))]\!]$ and $[\![\Iso (\sH (G))]\!]$, 
  \item[(iv)] $[\![\Mon (\sR (G))]\!]$ and $[\![\sH (G)]\!]$.
    \end{itemize}
 \end{thm}

\section{The Linear Theory} \label{s2}
\subsection{Structure of L-Representations}
Let $V$ be an irreducible L-representation. Each odd element $\vw$ yields an isomorphism of $\C G$-modules $\fn{\rho(\vw)}{V^*}{\vw \cdot V}$, where $\vw \cdot V$ is the vector space $V$ with twisted $G$ action, $\vg * v = (\vw \vg \vw^{-1})v$. It follows that the character $\chi$ of the underlying $\C G$-module $V$ satisfies $\overline{\chi} = \vw \cdot \chi$, where similarly, $(\vw \cdot \chi)(\vg) = \chi(\vw \vg \vw^{-1})$. The following result describes the structure of such $V$.

\begin{prop}\label{irrL}
One of the following mutually exclusive statements holds for an irreducible L-representation $V$.
\begin{itemize}
\item[(1)]
  $V\res{\C G}{} = W$ is a simple $\C G$-module; $W \cong \vw \cdot \overline{W}$ as $\C G$-modules; $W$ is of antilinear type $\bR$;
  $\Aut_L(V) = \{\pm \id\}$.
\item[(2)]
$V\res{\C G}{} = W \oplus W'$ is the sum of two simple $\C G$-modules, both of antilinear type $\bC$; $W \not\cong W'$ and  $W \not\cong \vw \cdot \overline{W}$ as $\C G$-modules; $\Aut_L(V) \cong \C \setminus 0$.
\item[(3)]
$V\res{\C G}{} = W \oplus W'$ is the sum of two simple $\C G$-modules, both of antilinear type $\bH$; $W \cong W'$ and $W \cong \vw \cdot \overline{W}$ as $\C G$-modules; $\Aut_L(V) \cong \SL_2(\C)$.
\end{itemize}
%
\end{prop}

\begin{proof}
  Let $W$ be a simple $\C G$-submodule of $V$. Since the form $B_{\vw}$ is $\vw$-invariant, $W^{\perp}$ is a $\C G$-submodule of $V$.
  Because the form is $\vw$-symmetric, $W^{\perp}={}^{\perp}W$ and $\ker (\left.B_{\vw}\right|_W) = W^{\perp} \cap W$ is a $\C G$-submodule of $W$. It must be zero or $W$. Hence, we have two mutually exclusive cases.

{\bf Case A: Some simple $\C G$-submodule $W$ of $V$ satisfies $W^{\perp} \cap W = 0$.}
It follows that $W$ is an L-subrepresentation of $V$, hence, $W = V$. By \cite[Prop. 2.15]{AREP}, $V$ has antilinear type $\bR$ and the bilinear form $B_\vw$ yields an isomorphism $W \cong \vw \cdot \overline{W}$.

Since $W$ is simple, any $\C G$-automorphism of $V$ is a scalar $\alpha \id$, $\alpha\in\C\setminus\{0\}$. 
The only scalars preserving the bilinear form $B_\vw$ are $\pm 1$, hence, $\Aut_L(V) = \{\pm \id\}$.
{\em This is statement~(1).}

{\bf Case B: Any simple $\C G$-submodule $W$ of $V$ satisfies $W^{\perp} \cap W = W$.}
Fix $W$. Notice that $W \subset W^{\perp}$ and write $V = W^{\perp} \oplus W'$ as a $\C G$-module.

Observe that $W \oplus W'$ is an L-subrepresentation. Indeed,
$$\ker (\left.B_{\vw}\right|_{W \oplus  W'}) = (W \oplus W')^{\perp} \cap (W \oplus W') \subseteq W^{\perp} \cap (W \oplus W') = W,$$
and, by the non-degeneracy of $B_{\vw}$ on $V$, $\ker (\left.B_{\vw}\right|_{W \oplus W'})  = 0$.
By the irreducibility of $V$, $V = W \oplus W'$ and, therefore, $W = W^{\perp}$.

It follows that $2 \dim W = \dim W + \dim W^{\perp} = \dim V$. Moreover, $W'$ is also a simple $\bC G$-module because any simple $\C G$-submodule $U$ of $W'$ also satisfies $2 \dim U = \dim V$.

The $\C G$-module homomorphism between simple $\bC G$-modules defined by $B_{\vw}$ 
\begin{equation} \label{hom_g}
g: \vw \cdot W \hookrightarrow {\vw \cdot V} \xrightarrow{f}  {V^*} \twoheadrightarrow {W'^*}, \ \  f(u)(v) = B_{\vw}(v,u),
\end{equation}
is non-zero because $W = W^{\perp}$.
Furthermore, $g$ determines $B_{\vw}$ because $B_{\vw}$ is $\vw$-symmetric:
\begin{align} \label{g_to_Bw}
B_{\vw}(u + u',v + v') = B_{\vw}(u , v') + B_{\vw}(u',v) = g(u)(\vw^{-2}v') + g(v)(u'),
\end{align}
for all $u + u',v + v' \in W \oplus W'$. We have two further subcases.

{\bf Subcase B1: $W \not\cong W'$.} Then $W \not\cong \vw \cdot \overline{W}$, so $W$ is of antilinear type $\C$ \cite[Thm. 5.4]{AREP}. Let $\fn{f}{V}{V}$ be a morphism of L-representations, $A$ -- the matrix of $f$. Note that $B(\vw) = A^T B(\vw) A$ for a fixed $\vw$ is a necessary and sufficient condition for $f\in \Aut_{\bC G} (V)$ to be a morphism of L-representations.
As $W \not\cong W'$ and $W = W^{\perp}$, $A$ and $B(\vw)$ have the form
\begin{align*}
B(\vw) =
\begin{pmatrix}
\mathbf{0} & Y \\ X & \mathbf{0}
\end{pmatrix}, \:
A =  
\begin{pmatrix}
\alpha \id & \mathbf{0} \\ \mathbf{0} & \beta\id 
\end{pmatrix}, \text{ thus }
\begin{pmatrix}
\mathbf{0} & Y \\ X & \mathbf{0}
\end{pmatrix}
=
\begin{pmatrix}
\mathbf{0} & \alpha \beta Y \\ \alpha\beta X &\mathbf{0}
\end{pmatrix}\, .
\end{align*}
It follows that $\alpha\beta = 1$ and $\Aut_L(V) \cong \C \setminus 0$.
{\em This is statement~(2).} 

{\bf Subcase B2: $W\cong W'$.} Then $W \cong \vw \cdot \overline{W}$.
To show that $W$ is of antilinear type $\bH$, 
it suffices to construct a non-degenerate, $\vw$-invariant, $\vw$-alternating bilinear form on $W$ \cite[Thm. 5.4]{AREP}. 

Let $\fn{h}{W}{W'}$ be a $\C G$-module isomorphism.
Consider $\widetilde{V}\coloneqq W \oplus W$.
Let us use the $\C G$-module isomorphism $\id\oplus h : \widetilde{V} \rightarrow V$
to turn $\widetilde{V}$ into an irreducible L-representation:
$$
\widetilde{B}_{\vw}((u_1,u_2),(v_1,v_2)) \coloneqq B_{\vw}(u_1 + h(u_2), v_1 + h(v_2))\, .
$$
Consider isomorphisms of $\bC G$-modules $\fn{f_1,f_2}{\vw \cdot W}{W^*}$ where
\begin{equation} \label{ff_equation}
f_1(w)(v) =  \widetilde{B}_{\vw}((0,v),(w,0)), \: f_2(w)(v) =  \widetilde{B}_{\vw}((v,0),(0,w)).
\end{equation}
There exists $\lambda \in \bC$ such that $f_1 = \lambda f_2$. As the form is $\vw$-symmetric, $f_2(w)(v) = f_1(\vw^2v)(w)$. Then 
\begin{align*}
  f_1(w)(v) &= \lambda f_2(w)(v) = \lambda f_1(\vw^2v)(w) = \lambda^2 f_2 (\vw^2v)(w), \\
  &= \lambda^2 f_1(\vw^2 w)(\vw^2 v) = \lambda^2 f_1(w)(v).
\end{align*}
Therefore, $\lambda = \pm 1$. If $\lambda = 1$, then the diagonal $\{ (w,w) \}$ is an L-subrepresentation of $\widetilde{V}$, which contradicts irreducibility. Thus, $\lambda = -1$ and the form $D(u,v) \coloneqq f_1(v)(u)$ on $W$ is non-degenerate, $\vw$-invariant and $\vw$-alternating. Thus, $W$ is of antilinear type $\bH$.
Moreover, this shows that $\widetilde{B}_{\vw}$ is symplectic.
As before, $B(\vw) = A^T B(\vw) A$ is a necessary and sufficient condition for $f\in \Aut_{\bC G} (V)$ to be a morphism of L-representations.
Now $A$ and $B(\vw)$ have the form
$$
B(\vw) =
\begin{pmatrix}
\mathbf{0} & -X \\ X & \mathbf{0}
\end{pmatrix}, \:
A =  
\begin{pmatrix}
\alpha \id & \beta \id \\ \gamma\id & \delta\id 
\end{pmatrix}, \text{ thus }
\begin{pmatrix}
\mathbf{0} & -X \\ X & \mathbf{0}
\end{pmatrix}
= (\alpha\delta-\beta\gamma) \begin{pmatrix}
\mathbf{0} & - X \\ X  & \mathbf{0}
\end{pmatrix}.
$$
Thus, $A$ is an L-homomorphism if and only if
$\begin{psmallmatrix}
\alpha & \beta \\ \gamma & \delta  
\end{psmallmatrix} \in \SL_2(\C)$.
Therefore, $\Aut_L(V) \cong \SL_2(\C)$.
{\em This is statement~(3).} 
\end{proof}

We can now describe all L-representations.

\begin{cor}
Any L-representation is determined up to isomorphism by the underlying $\C G$-module.
\end{cor}

\begin{proof}
By Maschke's Theorem (Proposition~\ref{Maschke_th}) and the Krull-Remak-Schmidt Theorem (Corollary~\ref{Maschke_cor}),
it is sufficient to prove the corollary for irreducible L-representations. Let $V$ be an irreducible L-representation, whose $\C G$-module carries a second L-representation structure, denoted $\widetilde{V}$ and $\widetilde{B}_{\vw}$. For $V$, consider the three mutually exclusive cases in Proposition~\ref{irrL}. 

In case~(1), $\widetilde{B}_\vw = \alpha B_{\vw}$ for some $\alpha\in\bC \setminus 0$, by Schur's Lemma. Any choice of the square root yields an isomorphism of L-representations $\sqrt{\alpha}\id:\widetilde{V}\rightarrow V$.

In cases~(2) and (3), $V = W \oplus W'$ as in case B of Proposition \ref{irrL}. Write $\widetilde{g}_\vw$ and $g_{\vw}$ for the maps from ~\eqref{hom_g}. Now with respect to both of the forms $\widetilde{B}_{\vw}$ and $B_{\vw}$, $W^{\perp} = W$, hence both $\widetilde{g}_\vw$ and $g_{\vw}$ are non-zero. Thus by Schur's Lemma we have that $\widetilde{g}_\vw = \alpha g_{\vw}$ for some $\alpha\in\bC \setminus 0$. Then by ~\eqref{g_to_Bw}, $\widetilde{B}_\vw = \alpha B_{\vw}$, and again $\sqrt{\alpha}\id:\widetilde{V}\rightarrow V$ is isomorphism of L-representations.
%
%
\end{proof}

The three mutually exclusive possibilities for an irreducible L-representation in Proposition~\ref{irrL} correspond exactly to antilinear types \cite[Table 4]{AREP}. In particular, any L-representation is an A-representation.

Now pick an A-representation $V$. Using the unitary trick, fix a $G$-invariant  hermitian form $\langle \cdot, \cdot \rangle$ on it.
Define
\begin{equation} \label{LRep_structure}
  B_{\vw}(u,v) \coloneqq \langle u , \vw^{-1} v \rangle + \langle v ,\vw u \rangle \, .
\end{equation}  
\begin{lem} \label{LRep_lemma}
Formula~\eqref{LRep_structure} defines an L-representation structure on $V$.
\end{lem}
\begin{proof}
  The forms are  non-degenerate. Consider $u\in V$ such that $B_\vw (u,v) = 0$ for all $v \in V$.
  Hence, $0 = B_\vw (u , \vw u) = \langle u , u \rangle + \langle \vw u ,\vw u \rangle$.
  Since $\langle w, w \rangle \geq 0$ for all $w \in V$, we conclude that $\langle u,u \rangle = 0$ and  $u = 0$.

It remains to verify the conditions of Lemma~\ref{lemma_3corners}. $B_\vw$ is $\vw$-invariant:
\begin{align*}
B_{\vw}(\vg u, \vw\vg\vw^{-1}v) &= \langle \vg u , \vg\vw^{-1}v \rangle + \langle\vw\vg\vw^{-1} v ,\vw \vg u \rangle, \\
&= \langle u , \vw^{-1}v \rangle + \langle v ,\vw u \rangle = B_{\vw}(u, v) \, .
\end{align*}
The second equality holds by $G$-invariance. $B_\vw$ is $\vw$-symmetric:
$$
B_{\vw}(u, \vw^{2}v) = \langle u , \vw v \rangle + \langle \vw^{2}v ,\vw u \rangle = 
\langle u ,\vw v \rangle + \langle v , \vw^{-1}u \rangle  = B_{\vw}(v, u) \, .
$$
The final condition holds as well:
\begin{align*}
B_{\vw_2}(u, \vw_2 \vw_{1}^{-1}v) &= \langle u , \vw_{1}^{-1}v \rangle + \langle\vw_{2}\vw_{1}^{-1} v ,\vw_{2} u \rangle, \\
&= \langle u , \vw_{1}^{-1}v \rangle + \langle v ,\vw_{1} u \rangle  = B_{\vw_1}(u, v) \, . \qedhere
\end{align*}
\end{proof}

Let us summarise the discussion in this section with the next result.
\begin{cor} \label{bijection_L}
A $\C G$-module extends to an A-representation if and only if it extends to an L-representation. This gives a bijection between isomorphism classes of A-representations and L-representations. 
\end{cor}

\subsection{Intermediate Category and Functors}\label{int_1}
Let $V_1, \ldots, V_k$ be a complete set of distinct irreducible A-representations.
We turn them into irreducible L-representations as in Lemma~\ref{LRep_lemma}, by choosing a $G$-invariant hermitian form $\langle \cdot , \cdot \rangle$ on each $V_i$.  

Let us consider the skeletons of $\sL(G)$ and $\sA(G)$ that consist of all finite direct sums $\oplus_{i=1}^k  n_i V_i$. Each object in this skeleton is canonically (after previous choices) equipped with a $G$-invariant hermitian form. Let use define {\em the intermediate category}  $\sA^{\ast}(G)$ on this skeleton: $\sA^\ast (G)(V,W)$ consists of those morphisms in $\sA(G)(V,W)$ that preserve the canonical hermitian form.
Such morphisms are necessarily injective. 
It follows from \eqref{LRep_structure} that any morphism in $\sA^\ast (G)(V,W)$ preserves all $B_\vw$, hence $\sA^\ast (G)(V,W)$
is a subset of $\sL (G)(V,W)$ as well.

Both inclusions define the same subspace topology on $\sA^\ast (G)(V,W)$. Thus, we have the following topological (enriched in $\sT$) functors
\begin{equation} \label{functors_AL}
  \sL(G) \xleftarrow{\Phi} \sA^\ast (G) \xrightarrow{\Psi} \Mon (\sA(G)) \, .
\end{equation}  
Both functors are essentially surjective on objects.  Hence, to prove parts (i) and (ii) of Theorem~\ref{main_theorem}, it suffices to show that the functors are homotopy equivalences on morphisms.

\subsection{Proof of Part (i) of Theorem~\ref{main_theorem}} \label{proof_part_1}
Since all morphisms are isomorphisms the task is to compute endomorphisms of an object
$U = \oplus_i  n_i V_i$
in all three categories. 
The direct sum decomposition works in $\C G$-modules. Hence,  
$$
\Aut_L (U) = \prod_i  \Aut_L (n_i V_i), \;
\Aut_A (U) = \prod_i  \Aut_A (n_i V_i),
$$
$$
\mbox{and } \ \ \ 
\Aut_{\ast} (U) \coloneqq \sA^*(G) (U,U) = \prod_i  \sA^*(G) (n_i V_i, n_i V_i),
$$
with similar direct product decompositions for the functors: $\Phi(U) = \prod_i  \Phi(n_i V_i)$ and $\Psi(U) = \prod_i  \Psi(n_i V_i)$.
This reduces the theorem to the case of an isotypical representation, that is, $U=nV$.

It suffices to show that both $\Phi(nV)$ and $\Psi(nV)$ are homotopy equivalences for an irreducible L-representation $V$ of dimension $d$.
These maps are homotopy equivalences between a Lie group and its maximal compact subgroup.
The proof consists of computations of these groups: all algebraic homomorphisms in the proof are continuous.
The result is summarised in Table~\ref{table_i}.
\begin{table}
\caption{Homotopy Equivalences for Part (i)}
\label{table_i}
\begin{center}
\begin{tabular}{ |c|c|c|c| }
Type of $V$  & $\Aut_L (nV)$ & $\Aut_{\ast} (nV)$ & $\Aut_A (nV)$\\
\hline  
$\bR$  & $\mO_n(\C)$ & $\mO_n (\bR)$ & $\GL_n (\bR)$ \\
$\bC$  & $\GL_n (\bC)$ & $\U_n (\bC)$ & $\GL_n (\bC)$ \\
$\bH$   & $\Sp_{2n} (\bC)$ & $\Sp(n)$ & $\GL_n (\bH)$ \\
\end{tabular}
\end{center}
\end{table}

{\bf Case 1:}
Let $V$ be an irreducible A-representation of type $\R$. This corresponds to $\Aut_A(V)=\bR$ and case~(1) in Proposition~\ref{irrL}.
Then $\Aut_A(nV)=\GL_n (\bR)$.

The automorphisms in $\sA^*(G)$ preserve a hermitian form so that $\Aut_{\ast} (nV) = \GL_n (\bR) \cap \U (nV)$. Choose an orthonormal basis of $V$. Extend it by repeating to an orthonormal basis of $nV$. In this basis, $\U (nV)$ consists of unitary matrices, while $\Aut_A(nV)$ consist of the $dn\times dn$-matrices $M=(M_{i,j})_{n\times n}$, where each block $M_{i,j}$ is $\alpha_{i,j}\id_d$, $\alpha_{i,j}\in\bR$. It follows  that $\Aut_{\ast} (nV) = \mO_n (\bR)$ is a maximal compact subgroup of $\Aut_A(nV)=\GL_n (\bR)$.
The groups have two components and $\Psi (nV)$ is a homotopy equivalence between a Lie group and its maximal compact subgroup.

The group $\Aut_L(nV)$ is a subgroup of $\Aut_{\bC G}(nV) = \GL_n(\C)$.
It consists of $M\in\GL_n(\C)\subseteq \GL_{dn}(\C)$ subject to the extra condition
\begin{equation}\label{cond_c1}
  M^T B^{nV}(\vw) M = B^{nV} (\vw) \, ,
\end{equation}
which could be checked for one element $\vw$. 
In the basis as above, the matrix $B^{nV}(\vw)$ is block-diagonal with $n$ square blocks $B^{V}(\vw)$.
On the other hand, $M\in\Aut_{\bC G}(nV)$ is of a special block structure as well: 
$M=(M_{i,j})_{n\times n}$, where $M_{i,j}=\alpha_{i,j}\id_d$, $\alpha_{i,j}\in\bC$.
Thus, condition~\eqref{cond_c1} becomes $(\alpha_{i,j})^T(\alpha_{i,j})=\id_n$ and $\Aut_L(nV) = \mO_n(\C)$.
Hence, $\Phi (nV)$ is a homotopy equivalence as well.

{\bf Case 2:}
Let $V$ be an irreducible A-representation of type $\bC$. This corresponds to $\Aut_A(V)=\bC$ and case~(2) in Proposition~\ref{irrL}.
In particular, $V = W \oplus W'$ as a $\bC G$-module and $\Aut_A(nV)=\GL_n (\bC)$.

To proceed, we choose an orthonormal bases of $W$ and $W'$ and replicate them through $nV$. This yields an explicit isomorphism
$$
\Aut_{\bC G} (nV) \rightarrow \GL_n (\bC) \times \GL_n (\bC), \ 
M = (M_{i,j})_{n \times n} \mapsto  \left( (\alpha_{i,j})_{n \times n}, (\beta_{i,j})_{n \times n} \right),
$$
where each block is a $d\times d$-matrix of the form
\begin{equation} \label{Mij_form}
M_{i,j}= \begin{pmatrix}
  \alpha_{ij} \id_{d/2} & \mathbf{0}_{d/2} \\ \mathbf{0}_{d/2} & \beta_{ij}\id_{d/2}
  \end{pmatrix} \, , \ \alpha_{ij},\beta_{ij}\in\bC \, .
\end{equation}
The subgroup $\Aut_A(nV)$ consist of the matrices $M$ with $\alpha_{ij}=\beta_{ij}$ for all $i$ and $j$.
Preservation of the hermitian form on $nV$ is equivalent to $M$ being hermitian, which, in turn, is equivalent to $(\alpha_{ij})_{n\times n}$ being hermitian.
Thus, $\Aut_{\ast} (nV) = \U_n (\bC)$ and $\Psi (nV)$ is a homotopy equivalence between a connected Lie group and its maximal compact subgroup.

Let us consider $M \in \Aut_L(nV) \leq \Aut_{\bC G}(nV)$. It must be in the form~\eqref{Mij_form}, additionally satisfying the condition~\eqref{cond_c1}.
In our basis as above, the matrix $B^{nV}(\vw)$ is block-diagonal with $n$ square blocks
$$
B^{V}(\vw) = 
\begin{pmatrix}
\mathbf{0}_{d/2} & X \\ Y & \mathbf{0}_{d/2}
\end{pmatrix} ,
$$
where $X$ and $Y$ are some invertible matrices. Thus, the condition~\eqref{cond_c1} becomes $(\alpha_{i,j})(\beta_{i,j})^T=\id_n$.
It follows that $M\mapsto (\alpha_{i,j})$ is an isomorphism $\Aut_L(nV) \cong \GL_n(\C)$ and $\Phi (nV)$ is a homotopy equivalence. 

{\bf Case 3:}
Let $(V,\rho)$ be an irreducible A-representation of type $\bH$. This corresponds to $\Aut_A(V)=\bH$ and case~(3) in Proposition~\ref{irrL}.
In particular, $V = W \oplus W$ as a $\bC G$-module and $\Aut_A(nV)=\GL_n (\bH)$.

To proceed, we choose an orthonormal basis $e_1,\ldots e_{d/2}$ of $W$. We choose $\vw W$ as the second direct summand in $W$ and use $\vw e_1,\ldots \vw e_{d/2}$ as an orthonormal basis there. We replicate this basis of $V$ through $nV$. This yields an explicit isomorphism
$$
\Aut_{\bC G} (nV) \rightarrow \GL_{2n} (\bC), \ 
M = (M_{i,j})_{n \times n} \mapsto  \Big( \begin{pmatrix}
  \alpha_{ij} & \beta_{ij} \\ \gamma_{ij} & \delta_{ij}
\end{pmatrix}_{2 \times 2} \Big)_{n \times n} \, ,
$$
where each block is a $d\times d$-matrix of the form
\begin{equation} \label{Mij_form3}
M_{i,j}= \begin{pmatrix}
  \alpha_{ij} \id_{d/2} & \beta_{ij}\id_{d/2} \\ \gamma_{ij}\id_{d/2} & \delta_{ij}\id_{d/2}
  \end{pmatrix} \, , \ \alpha_{ij},\beta_{ij},\gamma_{ij},\delta_{ij}\in\bC \, .
\end{equation}
The subgroup $\Aut_A(nV)$ consist of such matrices $M$ that
$\begin{psmallmatrix}\alpha_{ij} & \beta_{ij} \\ \gamma_{ij} & \delta_{ij}\end{psmallmatrix}$
belongs to a fixed copy of quaternions $\bH$ inside $M_2 ( \bC)$ for each block $M_{i,j}$.
In particular, $\Aut_A(nV)\cong\GL_n(\bH)$.

Preservation of the hermitian form on $nV$ is equivalent to $M$ being hermitian.
Since $(M_{i,j})^\ast = (M_{j,i}^\ast)$ and $M_{i,j}^\ast$ is quaternionic conjugation,
this is equivalent to $(M_{ij})_{n\times n}$ being hyperunitary, considered as quaternionic matrix.
Thus, $\Aut_{\ast} (nV) = \Sp (n)$ and $\Psi (nV)$ is a homotopy equivalence between a connected Lie group and its maximal compact subgroup.

As shown in the proof of Proposition~\ref{irrL}, the bilinear form $B^V_\vw$ on the L-representation is skew-symmetric.
Thus, $B^{nV}_\vw$ is non-degenerate and also skew-symmetric. 
Hence, after identifying $\Aut_{\bC G}(V^n)$ with $\GL_{2n}(\C)$, we get an isomorphism $\Aut_L(nV) \cong \Sp_{2n}(\C)$. 
Thus, $\Phi (nV)$ is a homotopy equivalence. \qed

\subsection{Proof of Part (ii) of Theorem~\ref{main_theorem}} \label{proof_part_2}
The proof is a natural continuation of the proof in Section~\ref{proof_part_1}.  
It reduces to the case of two isotypical representations $nV$ and $kV$, $k> n$. Indeed, if $k<n$ there are no morphisms, while
$k=n$ is done in Section~\ref{proof_part_1}. 
From the functors~\eqref{functors_AL} we get embeddings of the hom-spaces
\begin{equation} \label{functors_AL_1}
%
\sL(G)(nV,kV) \xleftarrow{\Phi(nV,kV)} \sA^\ast (G)(nV,kV) \xrightarrow{\Psi(nV,kV)} \Mon (\sA(G))(nV,kV) \, .
\end{equation}

The spaces admit the natural actions of the automorphisms groups by compositions on the right and on the left, for instance, 
$$
\Aut_{\ast} (kV) \times \sA^\ast (G)(nV,kV) \times \Aut_{\ast} (nV) \rightarrow \sA^\ast (G)(nV,kV)\, .
$$
By inspection, we will see that the spaces are homogeneous over $\Aut_{\ast} (kV)$. 
To complete the proof, we need to compute them explicitly. 

Let $\bK\in\{\bR,\bC,\bH\}$, $V$ -- an irreducible A-representation of type $\bK$.
Then $\Mon (\sA(G))(nV,kV))$ is the space of injective $\bK$-linear maps ${\bK^n}\rightarrow{\bK^k}$.
Replicating an orthonormal basis of $V$ to $nV$ (as in Section~\ref{proof_part_1})   
identifies $\sA^\ast (G)(nV,kV)$ with the $\bK$-Stiefel manifold.
The standard argument, based on the Gram-Schmidt process, proves that  $\Psi(nV,kV)$ is a homotopy equivalence
(see also Lemma~\ref{hom_eq}). 

The second map $\Phi(nV,kV)$ requires case-by-case considerations.

{\bf Case 1:}
Let $V$ be an irreducible A-representation of type $\bR$.
Proceeding similarly to case~1 in Section~\ref{proof_part_1}, 
we identify $\sL(G)(nV,kV)$ with the space of $k\times n$-matrices over $\bC$ with orthogonal columns.
By Witt's Extension Theorem, $\sL(G)(nV,kV)\cong \mO_k (\bC)/\mO_n (\bC)$.

Since $n<k$, we can restrict to matrices with the determinant $1$ so that the map $\Phi(nV,kV)$ becomes the natural embedding
of the homogeneous spaces
$$
\SO_k (\bR)/\SO_n (\bR) \cong \sA^{\ast }(G)(nV,kV)\xrightarrow{\Phi(nV,kV)} \sL(G)(nV,kV)\cong \SO_k (\bC)/\SO_n (\bC) \, . 
$$
It is a homotopy equivalence by the following standard fact, whose proof is similar to \cite[Theorem 4.14]{HJR}
\begin{lem} \label{hom_eq}
  Suppose that $H\subseteq G$ are connected Lie groups and $H$ is a closed subgroup.
  If $H_c \subseteq G_c$ are maximal compact subgroups of $H$ and $G$, then the natural map $G_c/H_c \rightarrow G/H$
  is a homotopy equivalence. 
\end{lem}
\begin{proof}
The commutative diagram of pointed connected topological spaces 
\begin{center}
  \begin{tikzcd}
(H_c,e) \arrow[d, "j_{H}"] \arrow[r, "f_c"]  & (G_c,e) \arrow[d, "j_{G}"]\arrow[r, "h_c"]  & (G_c / H_c, eH_c) \arrow[d, "j"]  \\
(H,e) \arrow[r, "f"]& (G,e) \arrow[r,"h" ] & (G/H, eH)
\end{tikzcd}
\end{center}
contains fibrations $h$ and $h_c$. Hence, it yields a map of long exact sequences of homotopy groups
\begin{center}
  \begin{tikzcd}
\cdots\rightarrow \pi_n(H_c) \arrow[d, "\pi_n(j_{H})"] \arrow[r, "\pi_n(f_c)"]  & \pi_n(G_c) \arrow[d, "\pi_n(j_{G})"]\arrow[r, "\pi_n(h_c)"]  & \pi_n(G_c / H_c) \arrow[d, "\pi_n(j)"]  \arrow[r, "\varpi_c"]  & \pi_{n-1}(H_c) \arrow[d, "\pi_{n-1}(j_{H})"]\rightarrow\cdots\\
\cdots\rightarrow \pi_n(H,e) \arrow[r, "\pi_n(f)"]& \pi_n(G) \arrow[r,"\pi_n(h)" ] & \pi_n(G/H)\arrow[r, "\varpi"] & \pi_{n-1}(H) \rightarrow\cdots
\end{tikzcd}
\end{center}
where all $\pi_n(j_{H})$ and $\pi_n(j_{G})$ are isomorphisms.
It follows that $\pi_n(j)$ are isomorphisms and $j$ is a weak homotopy equivalence. 
By Whitehead's Theorem, $j$ is a homotopy equivalence.
\end{proof}

{\bf Case 2:}
Let $V$ be an irreducible A-representation of type $\bC$.
Proceeding similarly to case~2 in Section~\ref{proof_part_1}, 
we identify $\sL(G)(nV,kV)$ with the space of $k\times n$-matrices over $\bC$ with unitary columns
so that the maps $\Phi(nV,kV)$ and $\Psi(nV,kV)$ are the same natural embeddings of the complex Stiefel manifolds
$$
\U_k (\bC)/\U_n (\bC) \cong \sA_{\ast }(G)(nV,kV)\xrightarrow{\Phi(nV,kV)} \sL(G)(nV,kV)\cong \GL_k (\bC)/\GL_n (\bC) \, . 
$$

{\bf Case 3:}
Let $V$ be an irreducible A-representation of type $\bH$.
Proceeding similarly to case~3 in Section~\ref{proof_part_1}, 
we identify $\sL(G)(nV,kV)$ with the space of $2k\times 2n$-matrices over $\bC$ whose columns form a Darboux basis of a subspace.
By (Symplectic) Witt's Extension Theorem, $\sL(G)(nV,kV)\cong \Sp_{2k} (\bC)/\Sp_{2n} (\bC)$.
The map $\Phi(nV,kV)$ becomes the natural embedding of the homogeneous spaces, which is a homotopy equivalence by Lemma~\ref{hom_eq}: 
$$
\Sp (k)/\Sp (n) \cong \sA^{\ast }(G)(nV,kV)\xrightarrow{\Phi(nV,kV)} \sL(G)(nV,kV)\cong \Sp_{2k} (\bC)/\Sp_{2n} (\bC) \, . 
$$
\qed

\section{The Hermitian Theory} \label{s3}

\subsection{Irreducible H-Representations}
Let $V$ be an irreducible H-representation. Each odd element $\vw$ yields an isomorphism of $\C G$-modules $\fn{\rho(\vw)}{\overline{V}^*}{\vw \cdot V}$, thus,  the character $\chi$ of the underlying $\C G$-module $V$ satisfies $\chi = \vw \cdot \chi$.

The following Proposition \ref{irrH} describing the structure of $V$ has one less case than the corresponding Proposition \ref{irrL} for the linear theory. This essential difference is due to the fact that $\vw$-invariant bilinear and sesquilinear forms behave differently under scaling.


Consider a $\vw$-invariant non-degenerate bilinear form $B$ on a simple $\C G$-module $W$. Then necessarily $B(u, \vw^2 v) = \lambda B(v,u)$ for $\lambda \in \bC^*$, because $f_1(u)(v) \coloneqq B(u, \vw^2 v)$ and $f_2(u)(v) \coloneqq B(v, u)$ both define non-zero elements of $\Hom_{\bC G}(\vw \cdot V, V^*)$. Furthermore, actually $ \lambda \in\{\pm 1\}$, because
\[
B(u,v) =  \lambda^{-1} B(v,\vw^{2}u) = \lambda^{-1}( \lambda^{-1} B(\vw^2 u, \vw^2 v) ) = \lambda^{-2} B(u,v).
\]
In this situation, for any $r \in \C$, the scaled bilinear form $\widetilde{B} \coloneqq rB$ also satisfies $\widetilde{B}(u, \vw^2 v) = \lambda \widetilde{B}(v,u)$. Conversely, if $B$ is sesquilinear, then similarly $B(u, \vw^2 v) = \lambda B(v,u)$, but for some $\lambda \in S^1$. For $r \in \C$, $\widetilde{B} \coloneqq rB$ satisfies instead 
$$
\widetilde{B}(u, \vw^2 v) = r \lambda \overline{B(v,u)} = \frac{r}{\overline{r}} \lambda \overline{\widetilde{B}(v,u)}.
$$
In particular, as $\lambda \in S^1$, we can always choose $r$ with $r^2 = \lambda^{-1}$ to make $\widetilde{B}$ $\vw$-symmetric.

\begin{prop}\label{irrH}
  Let $V$ be an irreducible H-representation.
One of the following mutually exclusive statements hold.
\begin{itemize}
\item[(1)]
$W \coloneqq V\res{\C G}{}$ is a simple $\C G$-module; $W \cong \vw \cdot W$ as $\C G$-modules; $\Aut_H(V) = \{\lambda \id \mid | \lambda | = 1\}$.
\item[(2)]
$V\res{\C G}{} = W \oplus W'$ decomposes as the sum of two simple $\C G$-modules; $W \not\cong W'$ and $W \not\cong \vw \cdot W$ as $\C G$-modules; $\Aut_H(V) \cong \C \setminus 0$.
\end{itemize}
\end{prop}

\begin{proof}
  The proof of \ref{irrL} remains true mutatis mutandis. The only major change is that when $V$ is the direct sum of two simple modules, $V = W \oplus W'$, necessarily $W \not\cong W'$.

  Indeed, suppose $h:W \xrightarrow{\cong} W'$ is an isomorphism.
  Construct an irreducible H-representation on $\widetilde{V}\coloneqq W \oplus W$ by transferring the structure via the $\C G$-module isomorphism $\id\oplus h : \widetilde{V} \rightarrow V$. 
  We have an isomorphism of H-representations as in \eqref{ff_equation} with $f_1 = \lambda f_2$ for some $\lambda \in S^1 \subseteq \bC$.
  Choose a square root $\eta \in S^1$,  $\eta^2 = \lambda$.
  Then $\langle (\eta w,w) \rangle_{\C}$ is an H-subrepresentation of $\widetilde{V}$, a contradiction.
\end{proof}

Unlike the linear theory, the underlying $\C G$-module does not always determine the H-representation.

\begin{cor}\label{determinetypeH}
Let $V$ be an irreducible H-representation. If $V$ is of type (1), then there are two non-isomorphic H-representations on $V\res{\C G}{}$ : these are $(V, B_{\vw})$ and $(V, - B_{\vw})$. If $V$ is of type (2), then $V$ is determined by $V \res{\C G}{}$. 
\end{cor}

\subsection{Relationship with $\C\G$-Modules}

Table~\ref{indextwoinduct} summarises the relationship between simple modules of $\C G$ and $\C \G$.
There $V$ is a simple $\C \G$-module; $W$ is a simple submodule of $V \res{\C G}{}$. 
\begin{table}
\caption{Index 2 Induction and Restriction}
\label{indextwoinduct}
\begin{center}
\begin{tabular}{ |c|c|c| } 
 $V \res{}{}$ & $W$ & $W \oplus \vw \cdot W$\\ 
 \hline
 $W \ind{}{}$ & $V \oplus (V \otimes \pi)$ & $V$ \\ 
 \hline
 $W \cong \vw \cdot W$? & Yes & No \\ 
 \hline
 $V \cong V \otimes \pi$? & No & Yes  \\ 
\end{tabular}
\end{center} 
\end{table}

\begin{cor} \label{bijection_H}
  A $\C G$-module is extendible to a $\C \G$-module if and only if it is extendible to an H-representation.
  This gives a bijection between isomorphism classes of $\C \G$-modules and H-representations. 
\end{cor}

\begin{proof}
  The bijection is given by the formula~\eqref{LRep_structure}.
  Note that if a $\C \G$-module $V$ defines $(V,B_{\vw})$, then $V \otimes \pi$ defines $(V,-B_{\vw})$.
  Finally, use Corollary~\ref{determinetypeH}.
\end{proof}

Notice that the bijection in Corollary~\ref{bijection_H} is natural in type~(2) but not in type~(1). In type~(1), $(V,B_\vw)$ can correspond to either $V$ or $V\otimes\pi$. Thus, we can talk about a natural $C_2^p$-torsor of bijections.

\subsection{Intermediate Category and Functors}\label{int_2}
Fix representatives of each isomorphism class of irreducible $\C \G$-modules: $V_1$, $V_1 \otimes \pi$, \ldots , $V_p$, $V_p \otimes \pi$, $U_1$, ... , $U_q$. Here $U_i \otimes \pi \cong U_i$ and $V_i \otimes \pi \not\cong V_i$.

We turn these into irreducible H-representations as in Lemma~\ref{LRep_lemma}, by choosing a $G$-invariant  hermitian form $\langle \cdot, \cdot \rangle$ on each. As in Section~\ref{int_1}, consider the skeletons of $\sH(G)$ and $\sR(G)$
consisting of all finite direct sums of the fixed irreducible modules.
Define {\em the intermediate category}  $\sR^{\ast}(G)$ on this skeleton: morphisms in $\sR^{\ast}(G)$ are those morphisms in $\sR(G)$ that preserve the chosen hermitian form.

We have the following topological (enriched in $\sT$) functors
\begin{equation} \label{functors_RH}
  \sH(G) \xleftarrow{\Phi} \sR^\ast (G) \xrightarrow{\Psi} \Mon (\sR(G)) \, ,
\end{equation}  
both of which are essentially surjective on objects.  Hence, as before, to prove parts (iii) and (iv) of Theorem~\ref{main_theorem}, it suffices to show that the functors are homotopy equivalences on morphisms.

\subsection{Proof of Part (iii) of Theorem~\ref{main_theorem}}\label{proof_part_3}

In the proof of part (i) it suffices to establish both homotopy equivalences for $nV$, where $V$ is an irreducible L-representation.
Here, in light of \ref{determinetypeH}, we instead need to show that both $\Phi(U)$ and $\Psi(U)$ are homotopy equivalences for $U = nU_i$ or $U = n V_i \oplus m (V_i \otimes \pi)$, as there are no $\C G$-morphisms between such $U$ for distinct $U_i$ or $V_i$.

The result is summarised in Table~\ref{table_iii}: $\Phi(U)$ and $\Psi(U)$ are always homotopy equivalences between a connected Lie group and its maximal compact subgroup.
\begin{table}
\caption{Homotopy Equivalences for Part (iii)}
\label{table_iii}
\begin{center}
\begin{tabular}{ |c|c|c|c| }
$U$  & $\Aut_H (U)$ & $\sR^\ast (G)(U,U)$ & $\Aut_R (U)$\\
\hline  
$n V_i \oplus m (V_i \otimes \pi)$  & $\U_{n,m}(\bC)$ & $\U_{n}(\bC) \times \U_{m}(\bC)$ & $\GL_n (\bC) \times \GL_m (\bC)$ \\
$n U_i$  & $\GL_n (\bC)$ & $\U_n (\bC)$ & $\GL_n (\bC)$ \\
\end{tabular}
\end{center}
\end{table}

{\bf Case 1:} Let $U = n V_i \oplus m (V_i \otimes \pi)$. The final two columns are standard classical results. Now let us show that $\Aut_H(U) = \U_{n,m}(\bC)$. 

The group $\Aut_H(U)$ is a subgroup of $\Aut_{\bC G}(U) = \GL_n(\C) \times \GL_m(\C)$.
It consists of $M\in \GL_n(\C) \times \GL_m(\C) \subseteq \GL_{dn}(\C) \times \GL_{dm}(\C)$ subject to the extra condition~\eqref{cond_c1} as before, as $M\in\Aut_{\bC G}(U)$ has
$M=(M_{i,j})_{(n+m) \times (n+m)}$, where $M_{i,j}=\alpha_{i,j}\id_d$, $\alpha_{i,j}\in\bC$.

In the basis as above, the matrix $B^{U}(\vw)$ is block-diagonal with $n$ square blocks: $B^{U}(\vw) = \text{diag}(B^{V_i}(\vw), ... , B^{V_i}(\vw), -B^{V_i}(\vw), ... , -B^{V_i}(\vw)) = (I_n \oplus -I_m)\: \text{diag}(B^{V_i}(\vw))$. Then $B^{U}(\vw) M = (I_n \oplus -I_m) \: \text{diag}(B^{V_i}(\vw)) M = (I_n \oplus -I_m)M \: \text{diag}(B^{V_i}(\vw))$, hence the condition~\eqref{cond_c1} is equivalent to $M \in \U_{n,m}(\C)$.
Therefore, $\Phi (U)$ is a homotopy equivalence. 

{\bf Case 2:} $U = n U_i$. The proof is identical to that of case 2 of the proof of part (i), with appropriate transposes changed to conjugate-transposes. \qed

\subsection{Proof of Part (iv) of Theorem~\ref{main_theorem}}\label{proof_part_4}

This is a continuation of the previous section. The proof reduces to showing that the embeddings of hom-spaces
\[
\sH(G)(U,U') \xleftarrow{\Phi(U,U')} \sR^\ast (G)(U,U') \xrightarrow{\Psi(U,U')} \Mon (\sR(G))(U,U'),
\]
are homotopy equivalences, where $(U,U')$ are in the following two cases.

{\bf Case 1:} $(U,U') = (n V_i \oplus m (V_i \otimes \pi) , n' V_i \oplus m' (V_i \otimes \pi))$ for $n \leq n', m \leq m'$. Replicating an orthonormal basis of $V_i$ to $nV_i$ and $V_i \otimes \pi$ to $m(V_i \otimes \pi)$ identifies $\sR^\ast (G)(U,U')$ on each isotypical component with a product of complex Stiefel manifolds.
The space $\Mon (\sR(G))(U,U')$ identifies with the product of the spaces of injective map $\bC^n\rightarrow\bC^{n'}$ and $\bC^m\rightarrow\bC^{m'}$. 
By the Gram-Schmidt argument, $\Psi(U,U')$ is a homotopy equivalence.

Similarly to case 3 of Section~\ref{proof_part_2}, $\Phi(U,U')$ is the natural embedding of homogeneous spaces, hence, a homotopy equivalence by Lemma~\ref{hom_eq}:

\[
\begin{tikzcd}
    \sR^{\ast }(G)(U, U') \arrow{r}{\Phi(U,U')} \arrow[d, phantom, "\rotatebox{-90}{$\cong$}"] & \sH(G)(nV,kV) \arrow[d, phantom, "\rotatebox{-90}{$\cong$}"] \\
    \U_{n'}(\bC) \times \U_{m'}(\bC)/\U_{n}(\bC) \times \U_{m}(\bC)  \arrow[r] & \U_{n',m'}(\bC)/\U_{n,m}(\bC).
  \end{tikzcd}
\]

{\bf Case 2:} $(U,U') = (n U_i, m U_i)$ for $n < m$. This is identical to case 2 of Section~\ref{proof_part_2}: both $\Psi(U,U')$ and $\Phi(U,U')$ are (different) natural embeddings of complex Stiefel manifolds. \qed

\section{Generalisations} \label{s4}
\subsection{Compact Groups}
The antilinear theory works equally well for a compact group $G$ instead of a finite group \cite{NEEB}.
In particular, all the results of our earlier paper \cite{AREP} remain valid under these assumptions.
Since irreducible continuous representations of compact groups are finite-dimensional, the results of the present paper remain valid as well.

\subsection{Infinite Dimensional Spaces}
It is subtle to replace a compact group with a locally compact group.
The antilinear irreducible representations are no longer finite-dimensional.
On the other hand, the linear and hermitian theory require a stronger form of duality, available only for finite-dimensional vector spaces.
One way around it is to consider Hilbert spaces and unitary representations instead.
It is interesting to investigate which of the results of the present paper would still hold in this case.

\subsection{General Coefficients}
Another interesting project, worth further attention, is to replace $\bC$ with a more general field or even a ring, as we have done in Section~\ref{s1.3}.
The antilinear theory (at least over a field) is probably attainable \cite{AREP}, although it will depend on a classification of graded division rings.
The formulation of the main results of the present paper will require some version of homotopy theory of schemes (cf. \cite{Voevod}). 

\subsection{General Gradings}
Suppose that $G \leq \G$ is a more general grading such that the quotient $\G/G$ is a Galois group of a field extension $\bF\leq\bK$.
The antilinear theory in this set-up is clear: these are just representations of the skew group ring $\bK\ast\G$.
It would be still interesting to develop the theory in full, in the spirit of \cite{AREP}. 
It is not clear to us how to approach the linear and hermitian theories in this context. 

\bibliography{lreps}{}
\bibliographystyle{plain}

\end{document}